\documentclass[12pt]{amsart}

\usepackage{latexsym, amssymb, amsmath}

\usepackage{amsfonts, graphicx}

\newtheorem{thm}{Theorem}

\newtheorem{rem}{Remark}

\begin{document}
\title[Conformal Positive Mass Theorems for Manifolds with Charge]{
Conformal Positive Mass Theorems for Manifolds with Charge}

\author{Qizhi Wang}
\address{School of Mathematics and Information Science\\ Guangzhou University \\ Guangzhou
510004, China} \email{qzwang@gzhu.edu.cn}

\begin{abstract}
  In this paper, we prove conformal positive mass theorems for asymptotically flat manifolds with charge. We apply conformal relations to show that if the conformal sum of scalar curvature is not less than the norm square of electric field and electric density, the sum of the mass will not less than the modulus of total electric charge. We also study the situation with inner boundary condition and manifolds with scalar charge.
\end{abstract}

\maketitle

\section{Introduction}
Positive mass theorem is a celebrated result in general relativity. It shows that the ADM mass \cite{B} is nonnegative for an asymptotically flat manifold satisfying the dominant energy condition. While conformal transformation is also important for many applications in physics, concerning the sum of the ADM masses of two conformal related metrics one can deduce conformal positive mass theorem.

Conformal positive mass theorem is applied in the context of static uniqueness of blackhole on asymptotically flat spacetimes\cite{M,MS}. In \cite{S}, it is proved by spinorial method, while it is proved by conformal relation of ADM masses in \cite{MS}. Later, some similar conformal positive mass theorems are proved in \cite{Wang, Wang-Tam}.

One common form of this type of theorem is the following.

\begin{thm}\label{1}
Let $(N , h)$ and $(N , h')$ be asymptotically flat Riemannian three-manifolds
with compact interior and finite mass, such that $h$ and $h¡ä$ are $C^{1,1}$ and related via the
conformal rescaling $h' = f^{2}h$ with a $C^{1,1}$ function $f > 0$. Assume further that there
exists a non-negative constant $\beta$ such that the corresponding Ricci scalars satisfy
$S + \beta f^{2}S' \geq 0$ everywhere. Then the corresponding masses satisfy $m + \beta m'\geq 0$.
Moreover, equality holds iff both $(N , h)$ and $(N , h'
)$ are flat Euclidean spaces.
\end{thm}
In \cite{MS}, the authors use charge as conformal factors to construct two complete asymptotically flat manifolds both with zero ADM mass. As the scalar curvature may not be nonnegative but their conformal sum is nonnegative, using the rigidity part of Theorem \ref{1}, the metrics are both flat, then they can deduce the conformal factor is a function of radial function, i.e., the metric and the conformal factor are spherically symmetric, thus the uniqueness of static blackhole can be proved.

In this paper, we will change our view on Theorem \ref{1}, we don't put charge as conformal factors, but we consider charge as one part of geometric invariants, and we get some conformal positive mass theorems. The theorems will show that though we don't have dominant energy condition for each manifold, their conformal sum will have some positivity property if some local sum of quantity is positive. Our theorems are applications of positive mass theorem for conformal manifolds, and they provide some interesting properties for ADM mass and charge.

In particular, we apply conformal relations of curvature to show that if the conformal sum of scalar curvature is not less than the norm square of electric field, the sum of the mass will not less than the total electric charge. In section 1, we treat the electric conformal positive mass theorem, and then extend to electromagnetic case in section 2,  we also study the situation with inner boundary condition with spinorial method in section, at last we study manifolds with scalar charge in section 5.

\section{Electric conformal positive mass theorem}
In this section, we will consider asymptotically flat manifolds with electric charge.

In general,  $(M^{n},g)$ is an asymptotically flat manifold \cite{B} means that $M\setminus K$ is diffeomorphic to $\mathbb{R}^{n}\setminus\mathbf{B}_{1}(0)$, where $K\subseteq M$ is compact and
\begin{equation}\label {1}
g_{kl}-\delta_{kl}\in C^{2,\alpha}_{\tau},
\end{equation}
the ADM mass on $(M,g)$ is
\begin{equation}\label {6}
m=\frac{1}{2(n-1)\omega_{n-1}}\lim\limits_{r\rightarrow\infty}\int_{S_{r}}(\partial_{i}g_{ij}-\partial_{j}g_{ii})\nu^{j}_{r}dvol_{S_{r}}.
\end{equation}

First we consider the time symmetric initial data $(M^{3},g,E)$ for Einstein-Maxwell equations where $E$ is a vector field on $M$ corresponding to the electric field. We assume that $M$ is asymptotically flat such that the ADM mass is well defined (with decay order $\tau>\frac{1}{2}$) and at infinity $E$ and $\partial E$ decay as $O(r^{-2})$ and $O(r^{-3})$, here $r=\mid x\mid$ for the chart in the asymptotic end.

We define the total electric charge as $$ Q=\frac{1}{4\pi}\lim_{r\rightarrow\infty}\int_{S_{r}}E\cdot ndS.$$

We assume that the electric charge density is zero which means that $\nabla\cdot E=0$, and the charged dominated condition is $R_{g}\geq2\mid E\mid^{2}$.
By the well-known positive mass theorem with charge\cite{GHH}, if $(M^{3},g,E)$ satisfies the charge dominated condition, then the ADM mass of $M$ should have $m_{g}\geq \mid Q\mid$.

The rigidity part is complicated by topological reason, in general the rigidity holds for a space slice of Majumdar-Papapetrou solutions\cite{KW}. For simplicity we just assume the manifold have only one cylindrical end, then $m_{g}= \mid Q\mid$ holds for manifold arises from an extreme $Reissner-Nordstr\ddot{o}m$ metric. The general rigidity theorem can be pursued similarly.

The general $Reissner-Nordstr\ddot{o}m$ metric is given by :
\begin{equation}
ds^{2}=f(r)^{2}dt^{2}-f(r)^{-2}dr^{2}-r^{2}d\Omega_{n-1}.
\end{equation}
Where $\Omega_{n-2}$ denote standard $(n-2)-$sphere, $f(r)=1-X+Y$ and
\begin{align*}
X=&\frac{2m}{r^{n-3}},\\
Y=&\frac{Q^{2}+P^{2}}{r^{2n-6}}.
\end{align*}
Here $Q$ and $P$ are total electric and magnetic charge respectively, and $m$ is the mass, in the following of this section we set $P=0$.

Since the rigidity only occurs for manifold with inner boundary, we will only consider them in section 5.

\begin{thm}
Let $(M , g)$ and $(M , g')$ be asymptotically flat Riemannian three-manifolds
with compact interior and finite mass, such that $g$ and $g'$ are $C^{2}$ and related via the
conformal rescaling $g' = e^{2f}g$ with a $C^{2}$ function $f > 0$. Let  $\overline{g}=e^{f}g$, the scalar curvature and charge satisfies
$S_{g} + e^{2f}S_{g'} \geq  4\mid E\mid_{g}^{2}$ and $\mid\nabla_{g}f\mid^{2}\geq\mid 3(e^{f})_{i}E^{i}\mid+\mid e^{f}div_{g}E\mid$ everywhere. Then the corresponding masses satisfy $m_{g} + m_{g'} \geq 2\mid Q\mid$.
\end{thm}

\begin{proof}
Motivated by the definition of asymptotically flat manifolds, see Corollary 2.1 of \cite{Wang-Tam}, $f=O_{2}(r^{-\tau})$ for $\tau>\frac{1}{2}$.
As $\overline{g}=e^{f}g$, by the conformal relation of scalar curvature\cite{Wang-Tam}, we have for dimension $n$,
\begin{equation}\label{2}
S_{\overline{g}}=e^{-f}(\frac{1}{2}S_{g}+\frac{1}{2}e^{2f}S_{g'}+\frac{1}{4}(n-1)(n-2)\mid\nabla_{g}f\mid^{2}).
\end{equation}
Since
\begin{equation}
2\mid E\mid_{\overline{g}}^{2}=2\overline{g}^{ij}E_{i}E_{j}=2e^{-f}g^{ij}E_{i}E_{j}=2e^{-f}\mid E\mid_{g}^{2},
\end{equation}

Let $n=3$ in (\ref{2}), as we assume $$\mid\nabla_{g}f\mid^{2}\geq\mid 3(e^{f})_{i}E^{i}\mid+\mid e^{f}div_{g}E\mid$$, since by direct calculation $div_{\overline{g}}E=div_{g}E+\frac{3}{2}f_{i}E^{i}$ then we have $$S_{\overline{g}}\geq 2\mid E\mid_{\overline{g}}^{2}+\mid div_{\overline{g}}E\mid$$ on $M$.

So the positive mass theorem with charge is satisfied for the initial data $(M^{3},\overline{g},E)$, also by $f=O_{2}(r^{-\tau})$, $Q_{g}=Q_{\overline{g}}=Q_{g'}$.

Thus we can set $Q=Q_{g}$ and $m_{\overline{g}}\geq\mid Q\mid$.

By the definition of ADM mass\cite{Wang}, we have $m_{g} + m_{g'}=2m_{\overline{g}}$ and then we get $m_{g} + m_{g'} \geq 2\mid Q\mid$.

\end{proof}

\begin{rem}
When the electric density for $\overline{g}$ vanish, the dominant charge condition reduce to $$S_{g} + e^{2f}S_{g'} \geq 4\mid E\mid_{g}^{2}$$,
by \cite{KW} the positive mass theorem can be applied in this dominant charge condition and the nonnegative result still holds.
As is well known, the rigidity of positive mass theorem with electric charge $\mid Q\mid=m\neq0$only occurs for manifold with inner boundary, so our theorem actually shows the rigidity holds only for $\mid Q\mid=m=0$ for complete manifold, and in this case by $\ref{2}$, we get $f=0$, so $(M , g)$ and $(M , g'
)$ are isometric to $\mathbb{R}^{3}$.
\end{rem}

\section{Electromagnetic conformal positive mass theorem}
In this section we extend the electric type theorem to the case with magnetic field.

We consider the time symmetric initial data for the time symmetric initial data $(M^{3},g,E,B)$ for Einstein-Maxwell equations where $B$ is a vector field on $M$ corresponding to the magnetic field. We always assume that $\nabla\cdot B=0$ for any metric and let $B$ and $\partial B$ decay as $O(r^{-2})$ and $O(r^{-3})$ and define the total magnetic charge as $$ P=\frac{1}{4\pi}\lim_{r\rightarrow\infty}\int_{S_{r}}B\cdot ndS.$$.

The charged dominated condition now becomes $S_{g}\geq 2\mid E\mid_{g}^{2}+2\mid B\mid_{g}^{2}+\mid div_{g}E\mid$.
The positive mass theorem will show $$m_{g}\geq \mid Q^{2}+P^{2}\mid^{\frac{1}{2}}$$ with this condition, and equality holds if and only if it arises from extreme $Reissner-Nordstr\ddot{o}m$ metric.

Since it is natural to assume magnetic density vanishes for a chosen metric, we can get by the same reasoning in section 2,
\begin{thm}
Let $(M , g)$ and $(M , g')$ be asymptotically flat Riemannian three-manifolds
with compact interior and finite mass, such that $g$ and $g'$ are $C^{2}$ and related via the
conformal rescaling $g' = e^{2f}g$ with a $C^{2}$ function $f > 0$. Assume the magnetic density for $\overline{g}=e^{f}g$ vanishes, the scalar curvature and charge satisfies $S_{g} + e^{2f}S_{g'} \geq  4\mid E\mid_{g}^{2}+4\mid B\mid_{g}^{2}$ and $\mid\nabla_{g}f\mid^{2}\geq\mid 3(e^{f})_{i}E^{i}\mid+\mid e^{f}div_{g}E\mid$ everywhere. Then the corresponding masses satisfy $m_{g} + m_{g'} \geq 2\mid Q^{2}+P^{2}\mid^{\frac{1}{2}}$.
\end{thm}

\begin{rem}
If we consider the initial data $(M,g,k,E,B)$ with second fundamental form, we still need to involve the linear momentum in the mass expression. In this situation, the charged dominant condition becomes more complicated\cite{Mc}, generally one needs $$\mu\geq\mid J\mid+2\mid E\mid_{g}^{2}+2\mid B\mid_{g}^{2}+\mid E\times B\mid_{g}.$$Where $(\mu,J)$ is energy-momentum tensor.

One can follow \cite{Wang-Tam} and the theorem above to get spacetime conformal positive mass theorems.
\end{rem}

\section{The inner boundary condition}
We extend the conformal positive mass theorems with inner boundary $\partial M$ in this section. As we assume the manifolds have only one cylindrical end for simplicity, we can let the boundary topologically a sphere.

Since the inner boundary is not necessarily an event horizon and we don't have the general positive mass theorem for manifold with boundary, we need to use spinor field on $M$ to take the boundary into consideration\cite{Wang}.

Let $(M,g)$ be the Riemannian n-dimensional manifold, ${e_{i}}$ the $g$- orthonormal frame, the spinors defined on M is denoted by $\sum M$. In fact, we can set $\sum M=SpinM\times_{\rho}\sum_{n}$ , where SpinM is a $Spin_{n}-principal$ fibre bundle over M and the spinor field $\sum_{n}\cong\mathbb{C}^{2^{[\frac{n}{2}]}}$. We denote the Levi-Civita connection for both the spinor bundle and the tangent bundle by $\nabla$, and the Dirac operator by $\mathfrak{D}=e^{i}\cdot\nabla_{i}$.

As usual, we need to solve the following Dirac system to deduce the mass expression.
\[
 \begin{cases}
\mathfrak{D}\psi&=0\\
P_{-}\psi\mid_{\partial M}&=0\\
\psi-\psi_{0}&\in{W^{2,p}_{-\tau}(\Gamma(\sum M))}
\end{cases}
\]
where $\psi_{0}$ is a constant spinor at infinity, $P_{\pm}$ are the $L^{2}$-orthogonal projections on the spaces of eigenvectors of positive (negative) eigenvalues of $\mathfrak{D}$ on $\partial M$, and
\begin{equation}
\|f\|_{W^{k,p}_{-q}(M)}:=(\int_{M} \sum\limits_{|\alpha|\leq k}(|D^{\alpha}f|\rho^{|\alpha|+q})\rho^{-n}dvol_{g})^{\frac{1}{p}}<\infty.
\end{equation}

In \cite{Wang}, we have
\begin{thm}
Let$(M,g)$ be asymptotically flat $n$ dimensional manifold with inner boundary $\partial M$ which is diffeomorphic to $S^{n-1}$, assume $g^{'}=e^{f}g$ on $M$ is also asymptotically flat. If $e^{f}S_{g^{'}}+S_{g}\geq0$ and $\frac{n-1}{2}H_{g}+\frac{n-1}{8}df(\nu)\leq\lambda_{0}$, then $m(g)+ m(g^{'})\geq0$. Here $\lambda_{0}$ is a lower bound of the first eigenvalue of Dirac operator on the boundary.
\end{thm}

\begin{rem}
We can take the lower bound \cite{Hi} as
\begin{equation}
\lambda\geq\frac{1}{2}\sqrt{\frac{n-1}{n-2}inf S\mid_{\partial M}}.
\end{equation}
and so we can denote
\begin{equation}
\lambda_{0}:=\frac{1}{2}\sqrt{\frac{n-1}{n-2}inf S\mid_{\partial M}}.
\end{equation}
\end{rem}

By \cite{Wang-Tam} we get
\begin{equation}\label {14}
 \begin{split}
 & -\nabla_{i}(\frac{2}{n-2}e^{-\frac{(n-2)f}{4}}(\nabla^{i}e^{\frac{(n-2)f}{4}})\mid \psi\mid^{2})\\
 =& \frac{2}{n-2}(\nabla_{i}e^{\frac{(n-2)f}{4}})e^{-\frac{(n-2)f}{2}}(\nabla^{i}e^{\frac{(n-2)f}{4}})\mid \psi\mid^{2}-\frac{2}{n-2}e^{-\frac{(n-2)f}{4}}(\nabla_{i}\circ\nabla^{i}e^{\frac{(n-2)f}{4}})\mid \psi\mid^{2}\\ & -\frac{2}{n-2}e^{-\frac{(n-2)f}{4}}(\nabla^{i}e^{\frac{(n-2)f}{4}})(\langle\psi,\nabla_{i}\psi\rangle+\langle\nabla_{i}\psi,\psi\rangle)\\
 =& \frac{2}{n-2}\mid\nabla_{i}\psi-e^{-\frac{(n-2)f}{4}}(\nabla_{i}e^{\frac{(n-2)f}{4}}) \psi\mid^{2}\\ &-\frac{2}{n-2}\mid\nabla_{i}\psi\mid^{2}+\frac{2}{n-2}e^{-\frac{(n-2)f}{4}}(\Delta_{g}e^{\frac{(n-2)f}{4}})\mid \psi\mid^{2}.
 \end{split}
\end{equation}

and the expression of $m(g^{'})-m(g)$:
\begin{equation}\label {28}
\begin{split}
m(g^{'})-m(g)=&\frac{1}{2\omega_{n-1}}\int_{\partial M}e^{\frac{(n-2)f}{4}}df(\nu_{r})dvol_{g}\partial M \\ &+\frac{1}{2(n-1)\omega_{n-1}}\int_{M}e^{\frac{(n-2)f}{4}}(e^{f}S_{g^{'}}-S_{g})dvol_{g}.
\end{split}
\end{equation}

By the above two relations, the sum of mass and the solvability of the Dirac system is due to the nonnegativity of the following calculation for spinor field $\psi$:

\begin{equation}\label {29}
\begin{split}
&\int_{S_{\infty}}\langle\psi,\frac{4}{n-1}\sigma_{ij}\nabla_{j}\psi-\frac{2}{n-2}e^{-\frac{(n-2)f}{4}}(\nabla_{i}e^{\frac{(n-2)f}{4}})\psi\rangle_{g}*_{g}e^{i}\\
&=\int_{M}(\frac{2}{n-2}\mid\nabla_{i}\psi+e^{-\frac{(n-2)f}{8}}(\nabla_{i}e^{\frac{(n-2)f}{4}}) \psi\mid^{2})dvol_{g}\\
 &+\frac{4}{n-1}\int_{\partial M}(\mathfrak{D}^{\partial}(\pi_{\ast}\psi)-(\frac{n-1}{2}H_{g}+\frac{n-1}{8}df(\nu))(\pi_{\ast}\psi),\pi_{\ast}\psi)dvol_{g}\partial M\\
 &+ \int_{M}((\frac{1}{2n-2}e^{f}S_{g^{'}}+\frac{1}{2n-2}S_{g})\mid\psi\mid^{2}\\
 &+(\frac{2n-6}{(n-1)(n-2)})\mid\nabla\psi\mid^{2}-\frac{4}{n-1}\mid\mathfrak{D}\psi\mid^{2})dvol_{g}.
\end{split}
\end{equation}
where $\sigma_{ij}=\delta_{ij}+e_{i}\cdot e_{j}\cdot$ and $\nu$ is the unit normal points toward inside of $M$, $D^{\partial}$ is Dirac operator on the inner boundary $\partial M$.

Using similar calculation, we get the following theorem:

\begin{thm}
Let $(M,g)$ be an asymptotically flat three dimensional manifold with inner boundary $\partial M$ which is diffeomorphic to $S^{2}$, assume $g^{'}=e^{f}g$ on $M$ is also asymptotically flat. If $\frac{1}{2}e^{f}S_{g^{'}}+\frac{1}{2}S_{g}-\mid E\mid^{2}_{g}-\mid div_{g}E\mid\geq0$ and $H_{g}+\frac{1}{4}df(\nu)+\mid g(E,\nu)\mid\leq\lambda_{0}$, then $m(g)+ m(g^{'})\geq2\mid Q\mid$. Here $\lambda_{0}$ is a lower bound of the first eigenvalue of Dirac operator on the boundary. Moreover, if the manifold has only one cylindrical end then equality holds iff both $(M , g)$ and $(M , g'
)$ are spatial extreme $Reissner-Nordstr\ddot{o}m$ metric with $m=\mid Q\mid$ outside the horizon.
\end{thm}
\begin{proof}

As in \cite{Wang} and (\ref{29}), we use a boundary term to relate $2m_{g}-2\mid Q\mid$ and $m_{g'}-m_{g}$ respectively in the left of the identity, and then express them by Stokes theorem. But notice that we have to introduce another connection $D$ to let the electric field come into the calculation, although with different connection, the spinor bundle are isomorphic and we can use the same symbol $\psi$ to represent a spinor field.

To make sense for the time symmetric initial data with electric charge, now we set $D_{i}=\nabla_{i}+A_{i}$ where $A_{i}=\frac{1}{2}E\cdot e_{i}\cdot$, and replace the $\nabla$ with $D$ in (\ref{29}) in the first term on the left of (\ref{29}), and note that the Dirac operator in this proof is related to the new connection $D$,  we can get

\begin{equation}\label {30}
\begin{split}
&\int_{S_{\infty}}\langle\psi,\frac{4}{n-1}\sigma_{ij}D_{j}\psi-\frac{2}{n-2}e^{-\frac{(n-2)f}{4}}(\nabla_{i}e^{\frac{(n-2)f}{4}})\psi\rangle_{g}*_{g}e^{i}\\
&=\int_{M}(\frac{2}{n-2}\mid \nabla_{i}\psi+e^{-\frac{(n-2)f}{8}}(\nabla_{i}e^{\frac{(n-2)f}{4}}) \psi\mid^{2})dvol_{g}\\
 &+\frac{4}{n-1}\int_{\partial M}(\mathfrak{D}^{\partial}(\pi_{\ast}\psi)-(\frac{n-1}{2}H_{g}+\frac{n-1}{8}df(\nu)+g(E,\nu))(\pi_{\ast}\psi),\pi_{\ast}\psi)dvol_{g}\partial M\\
 &+ \int_{M}((\frac{1}{2n-2}e^{f}S_{g^{'}}+\frac{1}{2n-2}S_{g}-\frac{1}{n-1}\mid E\mid^{2}_{g}-\frac{1}{n-1}div_{g}E)\mid\psi\mid^{2}\\
 &+(\frac{2n-6}{(n-1)(n-2)})\mid \nabla\psi\mid^{2}-\frac{4}{n-1}\mid\mathfrak{D}\psi\mid^{2})dvol_{g}.
\end{split}
\end{equation}

here we have used the following Lichnerowicz identity, the calculation of this identity can be deduced by the calculation (\ref{29}) and note we use the Bochner formula on the right the spinor field with connection $D$ (see\cite{KW}) :
$$(\mid \nabla\psi\mid^{2}+\frac{1}{4}\mathfrak{R}\mid\psi\mid^{2}-\mid\mathfrak{D}\psi\mid^{2})\ast1=d(\langle\psi,\sigma_{ij}\cdot D_{j}\psi\rangle\ast e_{i}).$$
Where $\mathfrak{R}=S_{g}-2\mid E\mid_{g}^{2}-div_{g}E,$ and the right side of this identity can be integrated to be mass minus the charge.

Let $n=3$, suppose that $\frac{1}{2}e^{f}S_{g^{'}}+\frac{1}{2}S_{g}-\mid E\mid^{2}_{g}-\mid div_{g}E\mid\geq0$ and the boundary term of (\ref{30}) nonnegative, by the asymptotic condition we see that $\mathfrak{D}$
is an isomorphism between weighted Sobolev spaces with Atiyah-Patodi-Singer boundary condition, thus the Dirac system can be solved \cite{H}.

Let $\psi$ be the solution of the Dirac system, we put it in  (\ref{30}), when $n=3$ and $\mid\psi_{0}\mid^{2}=1$ the left handside of (\ref{30}) becomes
$$8\pi (m_{g}-\mid Q\mid)+4\pi(m_{g'}-m_{g})=4\pi(m_{g}+m_{g'}-2\mid Q\mid)\geq0.$$

When equality holds, $(\ref{30})$ shows $\nabla f=0$, since $f=0$ at infinity, we get $f=0$ on $M$, thus $g$ and $g'$ are isometric, and our assumptions meet the requirement of positive mass theorem with charge, thus the rigidity follows by positive mass theorem with electric charge\cite{GHH}.

\end{proof}
\begin{rem}
We explain how the $\mid Q\mid$ come in the expression. Notice that we can assume the charge $Q$ be nonnegative first in the proof, if not, we can replace $E$ with $-E$, and the proof goes through.
\end{rem}
\section{Scalar field}     
Einstein-Scalar field system is one of the simplest non-vacuum system which arises in coupling gravity to a scalar field. The general scalar field $\phi$ satisfies $\triangle_{g}\phi=V'(\phi)$ for some $V(\phi)$. In this section we take the conformal scalar hair for static blackhole into consideration.

The energy-momentum tensor\cite{T} for conformal scalar field is $$T_{\mu\nu}=\nabla_{\mu}\phi\nabla_{\nu}\phi-\frac{1}{12}g_{\mu\nu}\mid\nabla\phi\mid^{2}-\frac{1}{6}(g_{\mu\nu}\nabla^{2}\nabla_{\mu}\nabla_{\nu}+G_{\mu\nu})\phi^{2}$$
Here $G_{\mu\nu}$ is the Einstein tensor and the scalar field $\phi=O(r^{-1})$ on a static spacetime $M\times\mathbb{R}$.

Since we don't have positive mass theorem such that we can relate ADM mass with the scalar field, we cannot get analogue of conformal positive mass theorems.
As the spacetime is static, $\phi$ is independent of time, thus by the constraint equations we can get the dominant energy condition as follows:
\begin{equation}
T_{00}=\frac{1}{6}S_{g}\phi^{2}-\frac{1}{4}\mid\nabla\phi\mid^{2}-\frac{1}{3}\phi\Delta_{g}\phi\geq0.
\end{equation}

Let $g'=\phi^{4}g$, by the conformal relation $S_{g'}=\phi^{-5}(-8)\Delta_{g}\phi+S_{g}\phi$, we get
\begin{equation}\label{c}
24T_{00}=3S_{g}\phi^{2}+S_{g'}\phi^{6}-6\mid\nabla\phi\mid^{2}\geq0.
\end{equation}

Since the dominant energy condition for manifold $(M,g)$ is $T_{00}=S_{g}\geq0$, and this is equivalent to $S_{g}\phi^{2}+ \frac{1}{3}S_{g'}\phi^{6}\geq2\mid\nabla\phi\mid^{2}$ by $(\ref{c})$. We get the weak positive mass theorem below:

\begin{thm}
Let $(M,g)$ be an asymptotically flat three dimensional manifold with conformal scalar field $\phi$, assume $g^{'}=\phi^{4}g$. If $S_{g}+ \frac{1}{3}S_{g'}\phi^{4}\geq2\phi^{-2}\mid\nabla\phi\mid^{2}$, then the ADM mass of $(M,g)$ is nonnegative. If mass vanishes, the metric must be flat without scalar field.
\end{thm}

In fact we can get more after integration of (\ref{c}) over $M$ , it is seen that if $S_{g}=0$ then $\nabla\phi=0$ by divergence theorem, as
$\phi=O(r^{-1})$, we can see that $\phi=0$. Thus scalar flat manifold does not arise from static spacetime with conformal scalar field.

\section*{\textbf{Acknowledgments}}
The author would like to thank professor Xie Naqing for his encouragement and support.


\begin{thebibliography}{99}
\bibitem{B}R. Bartnik,  The Mass of an Asymptotically Flat Manifold.
Commun. Pure. Appl. Math. \textbf{XXXIX} (1986) 661-693.
  \bibitem{GHH}G. Gibbons ,S. Hawking,G. Horowitz, M. Perry, Positive mass theorems for black holes. Communications in Mathematical Physics, 1983, 88(3):295-308.
      \bibitem{H}M. Herzlich, The positive mass theorem for black holes revisited. Journal of Geometry and Physics, 1998, 26(1-2):97-111.
 \bibitem{Hi}O. Hijazi, Lower bounds for the eigenvalues of the Dirac operator, J. Geom and Phy 16, 27-38, 1995.

  \bibitem{KW}M. Khuri,  and G. Weinstein , Rigidity in the positive mass theorem with charge. Journal of Mathematical Physics 54.9(2013):092501.
  \bibitem{MS} M. Mars , W. Simon, On uniqueness of static Einstein-Maxwell-dilaton
black holes, Adv. Theor. Math. Phys. 6, 279 (2003).
\bibitem{M}A.K.M. Masood-ul-Alam, Uniqueness proof of static charged dilaton black hole, Class. Quantum Gravity 10 (1993)
2649-2656.
\bibitem{Mc}S. Mccormick, On the Charged Riemannian Penrose Inequality with Charge Matter, arXiv:1907.07967v.
\bibitem{S}W. Simon, Conformal Positive Mass Theorems. Letters in Mathematical Physics, 2000, 50(4).
\bibitem{Wang-Tam}L. Tam , Q. Wang , Some conformal positive mass theorems, Differential Geometry and its Applications 53(2017):103-113.
 \bibitem{T}Y. Tomikawa , T. Shiromizu , and K. Izumi . On the uniqueness of the static black hole with conformal scalar hair. Progress of Theoretical and Experimental Physics 2017.3(2017).
  \bibitem{Wang}Q. Wang, Conformal positive mass theorems for asymptotically flat manifolds with inner boundary, Differential Geometry and its Applications, 2014, 33:105-116.
\end{thebibliography}
\end{document}